\documentclass[a4paper,10pt, leqno]{amsart}
\date{April 27, 2020}
\usepackage{amsmath, amssymb, graphics, setspace}
\usepackage[active]{srcltx}
\usepackage[dvips]{graphicx}
\usepackage[usenames]{color}
\title[Cuspidal edges  with the same first fundamental forms 
]{%
Cuspidal edges  with the same first fundamental forms 
along a  knot
}

\author{A.~Honda}
\address[Atsufumi Honda]{
Department of Applied Mathematics, 
Faculty of Engineering, Yokohama National University,
79-5 Tokiwadai, Hodogaya, Yokohama 240-8501, Japan
}
\email{honda-atsufumi-kp@ynu.ac.jp}

\author{K.~Naokawa}
\address[Kosuke Naokawa]{%
Department of Computer Science, 
Faculty of Applied Information Science,
Hiroshima Institute of Technology,  
2-1-1 Miyake, Saeki, Hiroshima, 731-5193, Japan
}
\email{k.naokawa.ec@cc.it-hiroshima.ac.jp}

\author{K. Saji}
\address[Kentaro Saji]{%
  Department of Mathematics,
  Faculty of Science,
  Kobe University,
  Rokko, Kobe 657-8501, Japan}
\email{saji@math.kobe-u.ac.jp}
\author{M. Umehara}
\address[Masaaki Umehara]{%
  Department of Mathematical and Computing Sciences,
  Tokyo Institute of Technology,
  Tokyo 152-8552, Japan}
\email{umehara@is.titech.ac.jp}
\author{K. Yamada}
\address[Kotaro Yamada]{%
  Department of Mathematics,
  Tokyo Institute of Technology,
  Tokyo 152-8551, Japan}
\email{kotaro@math.titech.ac.jp}


\thanks{%
The first author was partially supported by 
Grant-in-Aid for Early-Career Scientists
 No.~19K14526.
The second author 
was partially supported by 
Grant-in-Aid for Young Scientists (B), No.~17K14197,
and the third author
was 
partially supported by 
(C) No.\ 18K03301 from JSPS.
The fifth author 
was partially 
supported by the Grant-in-Aid for 
Scientific Research (B) No.\ 17H02839.
}%
\newcommand{\op}[1]{{\operatorname{#1}}}

\newcommand{\R}{\boldsymbol{R}}

\newcommand{\Z}{\boldsymbol{Z}}

\newcommand{\mc}[1]{{\mathcal #1}}
\newcommand{\mb}[1]{{\mathbf #1}}

\newcommand{\pmt}[1]{{\begin{pmatrix} #1  \end{pmatrix}}}

\renewcommand{\phi}{\varphi}
\renewcommand{\epsilon}{\varepsilon}

\renewcommand{\det}{\op{det}}

\numberwithin{equation}{section}
\newtheorem{theorem}{Theorem}[section]
\newtheorem{proposition}[theorem]{Proposition}
\newtheorem{corollary}[theorem]{Corollary}
\newtheorem{lemma}[theorem]{Lemma}

\theoremstyle{definition}
\newtheorem{defi}[theorem]{Definition}
\newtheorem{remark}[theorem]{Remark}
\newtheorem{example}[theorem]{Example}
\newtheorem*{acknowledgments}{Acknowledgments}
\begin{document}
\maketitle
\begin{abstract}
Letting $C$ be a compact $C^\omega$-curve embedded in $\R^3$
($C^\omega$ means real analyticity),
we consider a $C^\omega$-cuspidal edge $f$ along $C$.
When $C$ is non-closed,
in the authors' previous works,
the local existence of three distinct cuspidal edges along $C$
whose first fundamental forms coincide with that of $f$
was shown, under a certain reasonable assumption on $f$.
In this paper, if $C$ is closed, that is, $C$ is a knot,
we show that there exist infinitely many cuspidal edges 
along $C$ having the same first fundamental form as that of $f$
such that their images are non-congruent to each other, in general.
\end{abstract}

\section*{Introduction}
Let us first introduce some terminology.
By \lq$C^r$-differentiablity\rq\
we mean $C^\infty$-differentiability if $r=\infty$
and real analyticity if $r=\omega$.	

Let $f:U\to \R^3$ be a $C^r$-map
from a domain $U(\subset \R^2)$ into the Euclidean 
$3$-space $\R^3$.
A point $p\in U$ is called a {\it cuspidal edge} if
there exist 
\begin{itemize}
\item a local $C^r$-diffeomorphism $\phi$ 
from a neighborhood $V(\subset U$) of $p$ to 
a neighborhood of the origin of $\R^2$, and
\item a local $C^r$-diffeomorphism $\Phi$ 
from an open  subset of $\R^3$ containing $f(V)$ to a
neighborhood of the origin of $\R^3$
\end{itemize} \noindent
such that $\phi(p)=(0,0)$, $\Phi\circ f(p)=(0,0,0)$
and 
$$
\Phi\circ f\circ \phi^{-1}(u,v)=f_C(u,v),\qquad f_C(u,v):=(u^2,u^3,v).
$$
The map $f_C$ is the {\it standard cuspidal edge} whose image 
is indicated in Figure \ref{Fig:m3},~left.

\begin{figure}[htb]
\begin{center}
        \includegraphics[height=3.3cm]{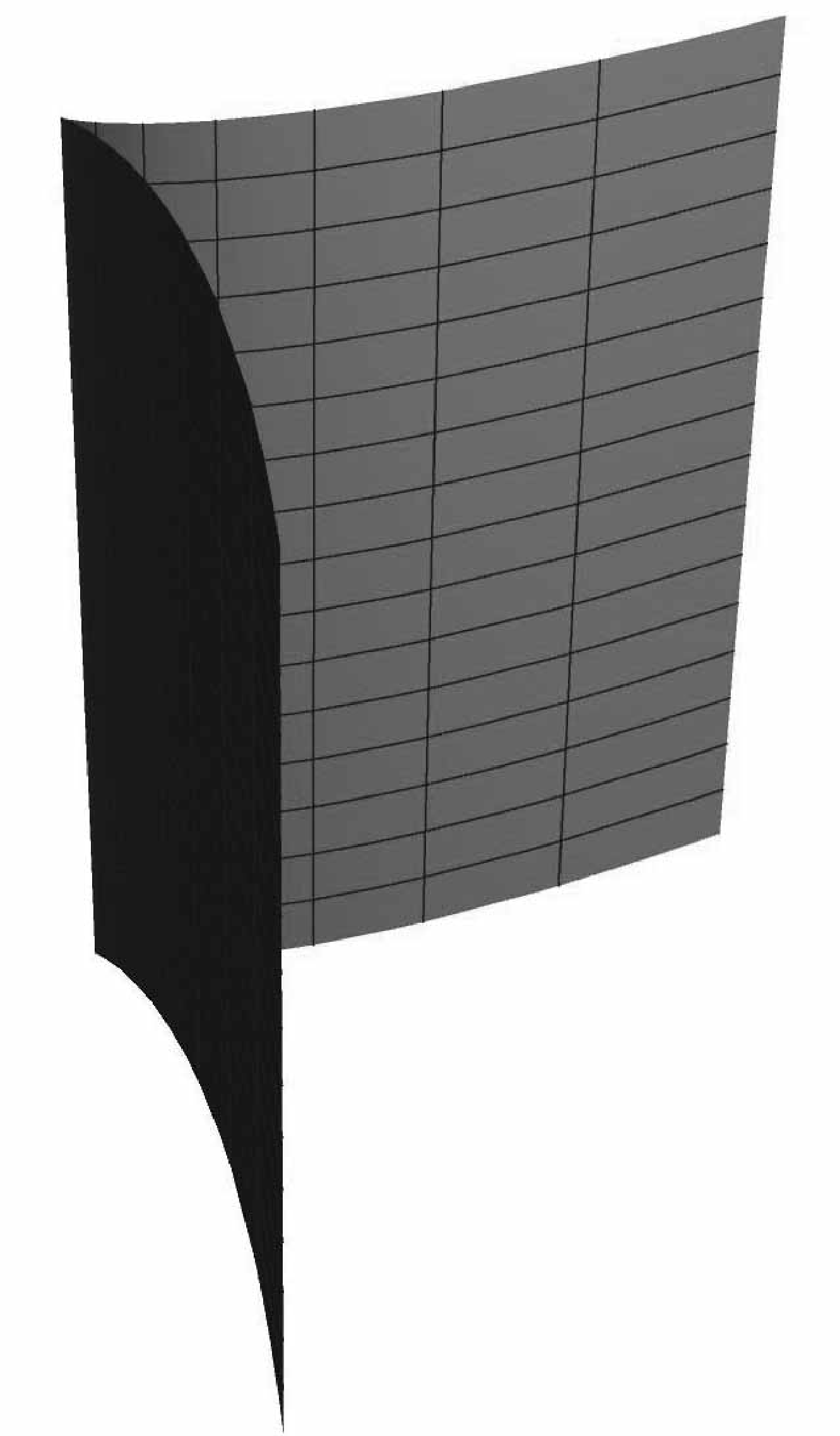}\quad \qquad \qquad
        \includegraphics[height=3.4cm]{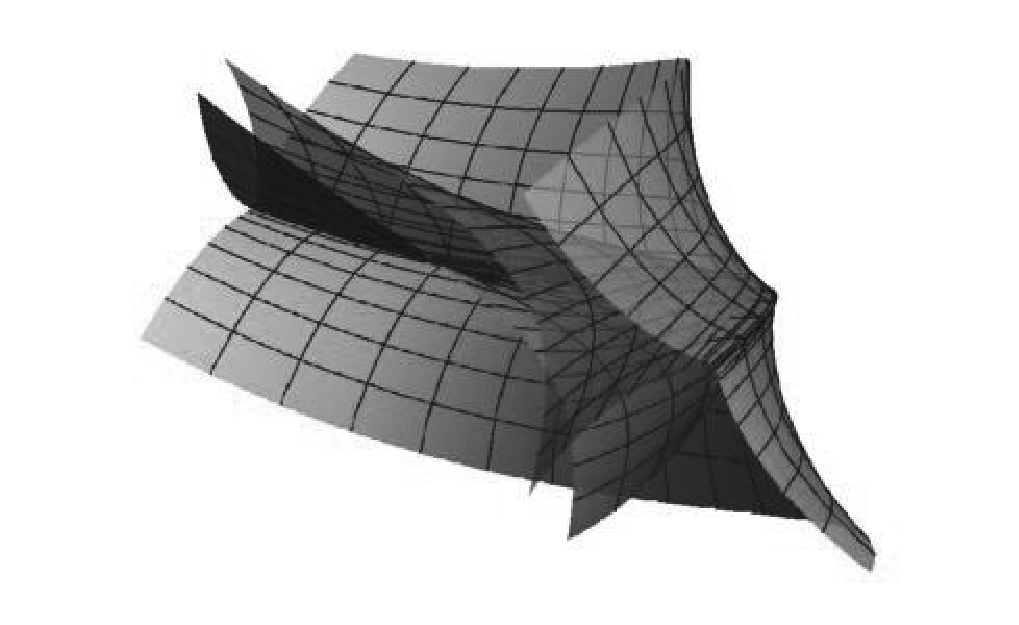}
\caption{The standard cuspidal edge and two 
mutually congruent cuspidal edges  along a helix}
\label{Fig:m3}
\end{center}
\end{figure}

We fix $l>0$, and denote by $I:=[0,l]$ the closed interval 
and by $S^1:=\R/l\Z$ the 1-dimensional torus of period $l$.
We set
$$
J:=I \quad \text{or}\quad S^1,
$$
since we treat the bounded closed interval $I$ and the one-dimensional
torus $S^1$ uniformly.
We then fix a $C^r$-embedded curve 
$\gamma:J \to \R^3$ with positive curvature
function, and denote by $C$ the image of $\gamma$.
For a positive number $\epsilon$, we set
$$
U_\epsilon(J):=J \times (-\epsilon,\epsilon).
$$

\begin{defi}
We say that $J_1$ is a {\it $J$-interval if
$J_1:=[0,a]$ when $J=I$, and $J_1:=\R/a\Z$ when $J=S^1$,
where $a>0$.}
\end{defi}

We fix a $J$-interval $J_1$.
A {\it $C^r$-cuspidal edge along $C$} is a $C^r$-map
$
\tilde f:U_\epsilon(J_1)\to\R^3
$
such that
\begin{itemize}
\item $J_1\ni t\mapsto \tilde f(t,0)\in \R^3$ gives a parametrization of $C$, 
and
\item $(t,0)$ is a cuspidal 
edge for each $t\in J_1$.
\end{itemize}
We denote by $f$ the map germ along $C$
induced by $\tilde f$.
For the sake of simplicity, we often identify $f$ with $\tilde f$,
if it creates no confusion.
(Later, we will give a special parametrization of $f$
(cf. \eqref{eq:repF}).
We denote by $\mc F^r(C)$ the set of germs of
 $C^r$-cuspidal edges along $C$.

\begin{defi}
Let $g:U_{\epsilon'}(J_2)\to \R^3$ ($\epsilon'>0$)
be a cuspidal edge along $C$, where $J_2$ is a $J$-interval.
Then $g$ is said to be {\it right equivalent} to $f$
if there exists a diffeomorphism $\phi$
from a neighborhood $U_1(\subset U_\epsilon(J_1))$
of $J_1\times \{0\}(\subset J_1\times \R)$
to a neighborhood $U_2(\subset U_{\epsilon'}(J_2))$
of $J_2\times \{0\}(\subset J_2\times \R)$
such that $\phi(J_1\times \{0\})=J_2\times \{0\}$ and
$f=g\circ \phi$ holds on $U_1$.
We denote by $[f]$ the right
equivalence class containing $f$.
\end{defi}

\begin{defi}
The cuspidal edge germ 
$g$ is said to be {\it isometric 
to $f\in \mc F^r(C)$
if there exists a diffeomorphism $\phi$
defined on a neighborhood of $J_1\times \{0\}(\subset J_1\times \R)$
such that $\phi(J_1\times \{0\})=J_2\times \{0\}$ and
the pull-back metric $\phi^*ds^2_g$
coincides with $ds^2_f$,
where $ds^2_f$ (resp. $ds^2_g$) is the first fundamental form
of $f$ (resp. $g$), that is, it is the pull-back of
the Euclidean inner product on $\R^3$ by $f$ (resp. $g$).
We denote this relationship by $g\sim f$.
When $g=f$,
such a $\phi$ is called a {\it symmetry} of $ds^2_f$
if $\phi$ is not the identity map.}
Moreover, if 
$$
\phi(t,0)=(t,0)\qquad (t\in J_1)
$$ 
holds, $\phi$ is said to be {\it non-effective}.
Otherwise, $\phi$ is called 
an {\it effective symmetry}.
\end{defi}

The isometric relation
defined above is an actual equivalence relation
on ${\mc F}^r(C)$, as well as being the right equivalence relation. 
The following assertion holds:

\begin{proposition}
Let $f,g\in {\mc F}^r(C)$.
If $g$ is right equivalent  to $f$, then
$g$ is isometric to $f$.
\end{proposition}

\begin{proof}
If $g$ is right equivalent  to $f$, then
there exists a local diffeomorphism $\phi$
such that $f=g\circ \phi$.
If we denote by $ds^2_E$ the Euclidean 
inner product of $\R^3$, then we have
that
$$
ds^2_f=f^*ds^2_E=(g\circ \phi)^*ds^2_E=\phi^*(g^*ds^2_E)=\phi^* ds^2_g,
$$
which proves the assertion.
\end{proof}

We then define an \lq\lq isomer" of $f\in {\mc F}^r(C)$ as follows.

\begin{defi}
For a given  $f\in {\mc F^r(C)}$,
a cuspidal edge $g\in {\mc F}^r(C)$
is called an {\it isomer} of $f$
(cf. \cite{HNSUY}) if it satisfies the 
following conditions:
\begin{enumerate}
\item $g$ is isometric to $f$ (i.e. $g\sim f$), but
\item $f$ is not right equivalent to $g$
(i.e. $[f] \ne [g]$).
\end{enumerate}
\end{defi}

A subset $A$ of $\R^3$  is said to be {\it congruent} to a subset $B(\subset \R^3)$
if there exists an isometry $T$ in $\R^3$ such that $B=T(A)$.
Moreover, we give the following definition:

\begin{defi}\label{d-c}
The image of a germ $g\in \mathcal F^r(C)$ 
is said to have the {\it same image} as 
a given  germ 
$f\in \mathcal F^r(C)$
if there exist open subsets $U_i$ ($i=1,2$) containing $J_i\times \{0\}$
such that $g(U_1)=f(U_2)$.
On the other hand,
$g$ is said to be {\it congruent} to $f$
if there exists an isometry $T$ of the Euclidean space $\R^3$
such that $T\circ g$ has the same image as $f$, as a map germ.
\end{defi}

\begin{remark}\label{rem:I}
Here, we consider the case $J=I$, that is, $C$ is non-closed.
Consider an \lq\lq admissible" $C^\omega$-germ of a cuspidal edge $f$ 
(i.e. $f$ belongs to the class $\mc F^\omega_*(C)$
defined in \eqref{eq:generic}).
If the first fundamental form of $f$ has no effective symmetries, then
there exist three distinct isomers 
$\check f, f_*, \check f_*\in \mc F^\omega(C)$
such that (cf. \cite{HNSUY})
\begin{itemize}
\item[$\diamond$] $\check f(t,v), f_*(t,v)$ and $\check f_*(t,v)$
have the same parameters as $f(t,v)$, and
\item[$\diamond$] the coefficients of the first fundamental forms of 
$\check f, f_*,\check f_*$ with respect to
the coordinate system $(t,v)$
coincide with those of $f$.
\end{itemize}
In fact, 
\begin{itemize}
\item 
$\check f$ (called the {\it dual} of $f$)
is the isomer whose cuspidal angle (cf. Definition \ref{def:angle}) 
takes the opposite sign of that of $f$,
\item 
$f_*$ (called the {\it inverse} of $f$)
is the isomer which is obtained by
reversing the orientation of 
the parametrization $\gamma$ of $C$.
The sign of the cuspidal angle of $f_*$ 
takes the same sign as that of $f$,
\item 
$\check f_*$ (called the {\it inverse dual} of $f$)
is the dual of the inverse $f_*$,
\item 
if $g$ is an isomer of $f$, then $g$ is right equivalent to
one of $\check f,\,\,f_*,\check f_*$.
\end{itemize}
The four maps $f,\check f, f_*, \check f_*$
in $\R^3$ 
are non-congruent in general. However, if $C$ admits a symmetry,
this is not true.
For example, consider a helix $C_0$ and fix a point $P_0$ on $C_0$.
Then there exists an orientation preserving isometry $T$ 
of $\R^3$ (which is not the identity map)
satisfying $T(C_0)=C_0$ and $T(P_0)=P_0$.
Then, we can construct a cuspidal edge along $C_0$ 
(cf. Remark 2.3)
such that $g:=T\circ f$ is
 an isomer of $f$.
The images of such a pair $(f,g)$ are indicated in Figure \ref{Fig:m3}, right.
\end{remark}

From now on, we consider the case that $J=S^1$, that is,
$C$ is a knot in $\R^3$,
and show that each $f\in \mc F^\omega(C)$
has infinitely many isomers 
which are mutually non-congruent, in general.

\begin{figure}[htb]
\begin{center}
        \includegraphics[height=3.3cm]{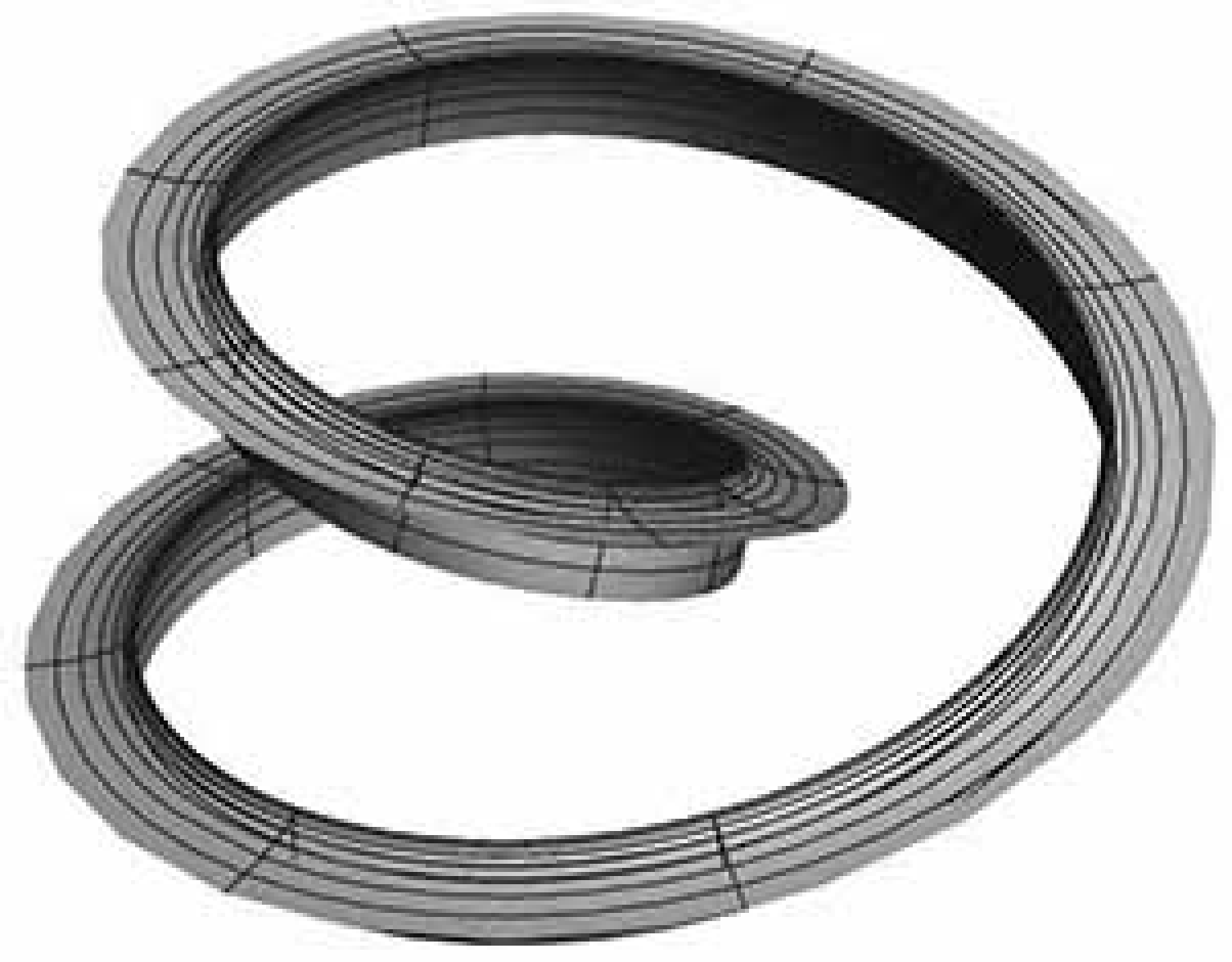}\qquad \qquad
        \includegraphics[height=3.9cm]{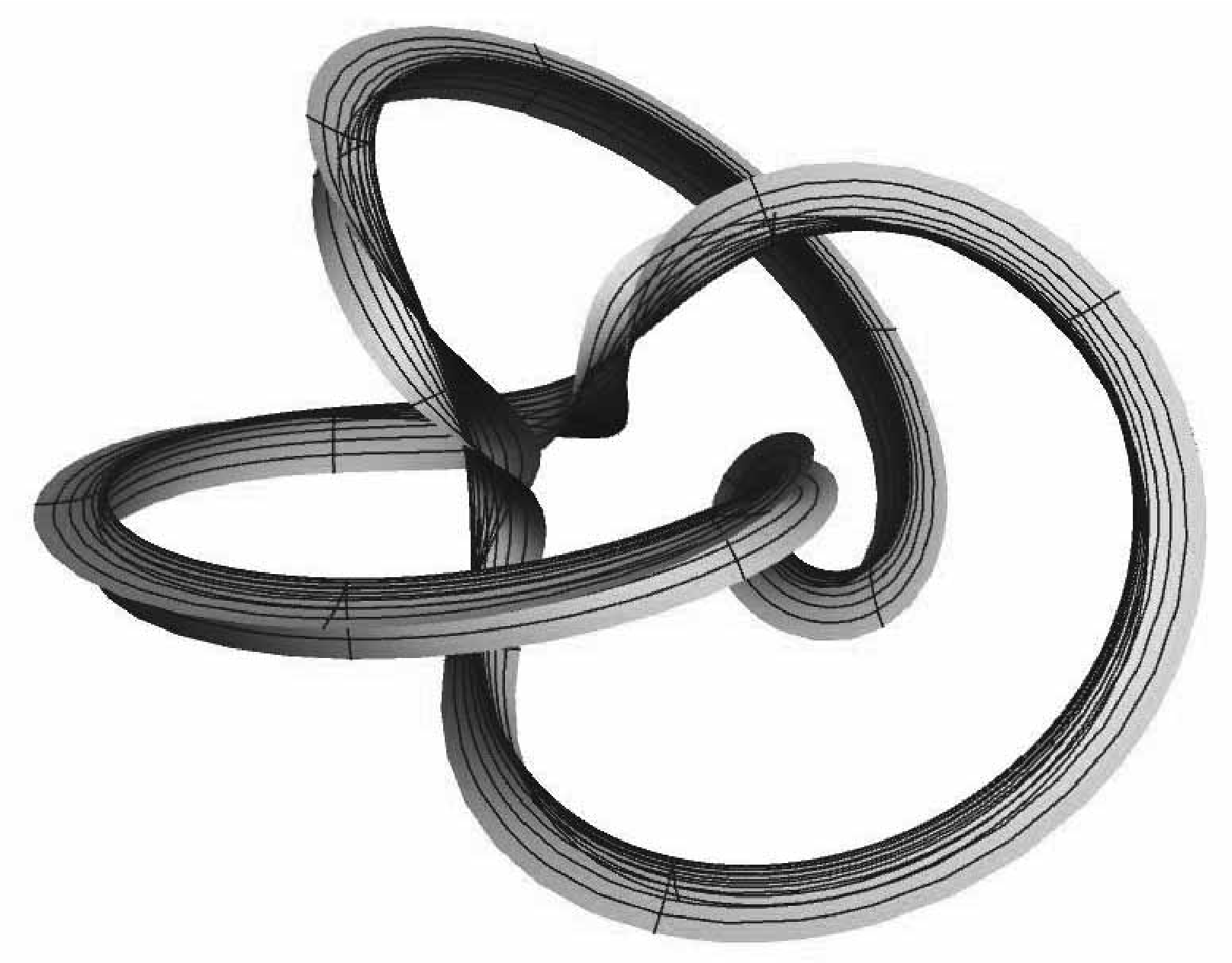}
\caption{Cuspidal edges along the  curves 
$\gamma_1$ (left) and $\gamma_2$ (right) (cf. Example \ref{eq;k})}
\label{Fig:m3b}
\end{center}
\end{figure}

\section{Results}
We set $J:=S^1$ and consider the case that 
$C(=\gamma(S^1))$ is a closed $C^r$-embedded curve with 
positive curvature function $\kappa(t)$.
We let $\mb n(t)$ (resp. $\mb b(t)$) be the unit principal 
normal (resp. unit bi-normal)
vector of $\gamma(t)$. 
We set 
$$
P_0:=\gamma(0),
$$
which is considered as a base point of $C$.
The parametrization $\gamma$ of $C$ 
gives an orientation of $C$.
For this fixed base point $P_0$ and this fixed orientation of
$C$, we would like to show that any cuspidal edge germ along $C$
can be uniquely represented using a normal form 
given as follows:

For sufficiently small  $\epsilon>0$, consider a $C^r$-map
$($called {\it Fukui's formula}, cf. \cite{F} and \cite{HNSUY}$)$
$f(t,v)$ ($(t,v)\in U_\epsilon(S^1)$)
expressed by
\begin{equation}\label{eq:repF}
f(t,v):=\gamma(t)+
(A(t,v),B(t,v)) \pmt{
\cos \theta(t) & -\sin \theta(t) \\
\sin \theta(t) & \cos \theta(t) 
}\pmt{\mb n(t) \\ \mb b(t)}
\end{equation}
such that 
\begin{itemize}
\item[(a)] $f(0,0)=P_0$,
\item[(b)] $A(t,v)$, $B(t,v)$ and $\theta(t)$ are $C^r$-functions, and
\item[(c)] for each $t\in S^1$, 
$A_{vv}(t,0), B_{vvv}(t,0)$ are not equal to zero.
\end{itemize}
In this setting, the angle $\theta(t)$ in \eqref{eq:repF}
is called the {\it cuspidal angle} at $\gamma(t)$.

\begin{defi}\label{def:angle}
A map $f(t,v)$  satisfying (a), (b) and (c)
is called a {\it normal form} of the cuspidal edge along $C$
with respect to the base point  $P_0$
if 
\begin{itemize}
\item[(d)]  $t$ is an arc-length parameter of $C$,
\item[(e)]  
for each $t\in S^1$, $v$ is the normalized half-arc-length parameter
of the sectional cusp $v\mapsto (A(t,v),B(t,v))$
 (see \cite[Appendix A]{HNSUY}
 for the definition of 
normalized half-arc-length parameters), that is,
there exists a $C^r$-function $m(t,v)$ satisfying  $m(t,0)\ne 0$
such that
$A(t,v)$ and $B(t,v)$ have the following expressions:
$$
(A(t,v),B(t,v)):=\int_0^v w\Big(\cos \lambda(t,w),\sin \lambda(t,w)\Big) dw,
\,\,\, \lambda(t,v):=\int_0^v m(t,w)dw.
$$
\end{itemize}
The function $m(t,v)$ is called the {\it extended half-cuspidal
curvature function} (cf. \cite{HNSUY}).
In this situation, the {\it singular curvature} $\kappa_s(t)$
and the {\it limiting normal curvature} $\kappa_\nu(t)$
(cf. \cite{SUY}) along the
singular set of $f\in \mc F^r(C)$  
are given by (cf.~\cite{HNSUY})
\begin{equation}\label{eq:kskn}
\kappa_s(t)=\kappa(t)\cos\theta(t),\qquad
\kappa_\nu(t)=\kappa(t)\sin\theta(t).
\end{equation}
\end{defi}

The following assertion holds:

\begin{proposition}\label{prop:000}
For each  $f\in {\mc F^r(C)}$, 
there exists a unique normal form $\hat f\in {\mc F^r(C)}$
with respect to the base point $f(0,0)$
such that
\begin{itemize}
\item $[f]=[\hat f]$, 
and
\item  
the orientation of $C$ given by
the parametrization $t\mapsto \hat f(t,0)$
coincides with that induced by
the parametrization $t\mapsto f(t,0)$.
\end{itemize}
\end{proposition}

\begin{proof}
The uniqueness of such an $\hat f$ follows 
from the fact that $\hat f(0,0)=f(0,0)$ and $\hat f(t,0)=\gamma(t)$,
since $t$ is an arc-length parameter of $\gamma$
and $v$ is the normalized half-arc-length parameter
of the sectional cusps of $\hat f$.
\end{proof}

We prepare a lemma:

\begin{lemma}\label{lem:0}
Let $f,g\in \mc F^r(C)$ be 
two normal forms of cuspidal edges along $C$.
If the image of $f$ coincides with that of $g$ and
$f(t,0)=g(t,0)$ holds for $t\in S^1$,
then either $f(t,v)=g(t,v)$ or $f(t,v)=g(t,-v)$ holds.
\end{lemma}

\begin{proof}
Suppose that the images of the two maps coincide.
Since $f$ and $g$ are written in normal forms,
the fact that $t$ is an arc-length parameter and
$v$ is a half-arc-length parameter implies 
$$
g(t,v)=f(\sigma t+a,\sigma'v) \qquad (t\in S^1),
$$
where $a\in S^1$ and $\sigma,\sigma'\in \{1,-1\}$.
Since $f(t,0)=g(t,0)$ holds for $t\in S^1$,
we have $\sigma=1$ and $a=0$, proving the assertion.
\end{proof}

As a consequence, the following assertion holds:

\begin{proposition}\label{prop:eq}
Let $f,g\in \mc F^r(C)$ be two cuspidal edge germs.  
Then $g$ is right equivalent to $f$ if and only if
$g$ has the same image as $f$.
\end{proposition}

\begin{proof}
The \lq\lq if"-part is obvious.
So it is sufficient to show the \lq\lq only if"-part.  
By Proposition \ref{prop:000},
there exist normal forms $\hat f$ and $\hat g$ of cuspidal edges
along $C$ such that 
$$
[f]=[\hat f],\quad  [g]=[\hat g],\quad \hat f(0,0)=\hat g(0,0).
$$
We suppose that $g$ has the same image as $f$.
Then $\hat g$ also has the same image as $\hat f$.
Replacing $\hat g(t,v)$ by $\hat g(-t,v)$ if necessary,
we may assume that
$\hat g(t,0)=\hat f(t,0)$ for $t\in S^1$.
Thus, by Lemma \ref{lem:0}, 
we have $\hat g(t,v)=\hat f(t,\pm v)$, which implies $[g]=[f]$.
\end{proof}

\begin{corollary}\label{cor:X}
Let $f,g\in \mc F^r(C)$ be two cuspidal edge germs.  
If $g$ is congruent to $f$, then 
there exists an isometry $T$ of $\R^3$
such that $T\circ g$ is right equivalent to $f$.
\end{corollary}

\begin{proof}
Suppose that $g$ is congruent to $f$.
Then, there exists an isometry $T$ of $\R^3$ 
such that $T\circ g$ has the same image as $f$.
By Proposition \ref{prop:eq},
we can conclude that
$T\circ g$ is right equivalent to $f$.
\end{proof}

We fix $f\in \mc F^r(C)$.
Then there exists a normal form $\hat f$ of the 
cuspidal edge along $C$
such that $[\hat f]=[f]$.
The expression $\hat f(t,v)$ means that $t$ is 
the arc-length parameter of $C$ and
$v$ is the normalized half-arc-length parameter of the
sectional cusps.
Let
$\kappa_s:S^1\to \R$
be the singular curvature function
of $\hat f$ along $C$.
By \eqref{eq:kskn},
$|\kappa_s(t)|\le \kappa(t)$ holds,
and  $\kappa_s(t)$ depends only on the first 
fundamental form of $f$ (cf. \eqref{eq:2} and \cite{SUY}).
We then consider the condition
\begin{equation}\label{eq:mM}
\max_{t\in S^1}|\kappa_s(t)| < \min_{t\in S^1}\kappa(t),
\end{equation}
and define the subclass
\begin{equation}\label{eq:generic}
\mc F^r_*(C):=\{f\in \mc F^r(C)\,;\, 
\mbox{$f$ satisfies \eqref{eq:mM}}
\}
\end{equation}
of $\mc F^r(C)$. 
A germ of a cuspidal edge $f\in \mc F^r(C)$
is called {\it admissible} if it belongs to
this subclass $\mc F^r_*(C)$.
For $f\in \mc F^r_*(C)$, we may assume that
its cuspidal angle $\theta(t)$ satisfies
\begin{equation}\label{eq:angle}
0<|\theta(t)|<\pi\qquad (t\in S^1).
\end{equation}
It should be remarked that cuspidal edges with
constant Gaussian curvature 
satisfy $\kappa_\nu=0$, that is, $|\kappa_s|=\kappa$ on $S^1$.
In particular, such surfaces do not 
belong to $\mc F^r_*(C)$.
(If $f$ is of constant Gaussian curvature, then 
$\kappa_\nu$ vanishes identically. For such a case, see \cite{B}.)

\begin{example}\label{eq;k}
Consider the following $2\pi$-periodic curves
giving a series of torus knots ($n:=2m-1$, $m=1,2,3,...$):
\begin{equation}\label{eq:cm}
\gamma_m(t):=\biggl( (2+\cos n t)\cos 2t, (2+\cos n t)\sin 2t,\sin n t\biggr)
\quad (t\in \R),
\end{equation}
and denote by $C_m$ their images.
The curve $C_2$ gives a trefoil knot.
The cuspidal edges $f_m$ for $m=1,2$  along $C_m$
obtained from \eqref{eq:repF} by substituting  
$$
A(t,v)=t^2,\quad B(t,v)=t^3,\quad \theta=\pi/4
$$ are indicated in Figure 2. 
Each of $f_m$ ($m\ge 1$)
belongs to the class $\mc F^r_*(C_m)$ as a map germ,
since we have chosen $\theta$ so that $0<\theta<\pi/2$.
\end{example}

\medskip
\begin{defi}\label{def:curveS}
We say that $C$ has a {\it symmetry} if 
there exists an isometry $T$ of the Euclidean space $\R^3$ 
such that $T(C)=C$ and $T$ is not the identity map.
On the other hand,
a $C^r$-function $\mu:S^1(=\R/l\Z)\to \R$ is said to have
a {\it symmetry} if there exists a constant $c\in (0,l)$
or $c'\in [0,l)$
 such that
$$
\mu(t)=\mu(t+c) \quad \text{or} \quad \mu(t)=\mu(c'-t)
$$
holds for $t\in \R$.
\end{defi}

If $C$ has a symmetry $T$, and if there is a point $P\in C$
such that $T(P)\ne P$, then $\kappa(t)$
also admits a symmetry. 
Our main result is as follows:

\medskip
\begin{theorem}\label{thm:main}
Let $C$ be the image of a closed $C^\omega$-curve embedded in $\R^3$,
and let $g$ be a cuspidal edge germ belonging to $\mc F^\omega_*(C)$. 
Then there are four continuous $1$-parameter
families of real analytic cuspidal edges 
$
\{f^i_{P}\}_{P\in C}
$
$(i=1,2,3,4)$ which satisfy the following properties:
\begin{enumerate}
\item[(i)]
Each $f^i_{P}$ $(i\in \{1,2,3,4\},\, P\in C)$
belongs to $\mc F^\omega_*(C)$ and is isometric to $g$. 
Moreover, there exist $i_0\in \{1,2,3,4\}$ and $P_0\in C$
such that $f^{i_0}_{P_0}$ is right equivalent to $g$.
\item[(ii)]
If $h\in \mc F^\omega_*(C)$ is an isomer of $g$, then
$h$ is right equivalent to a cuspidal edge germ
belonging to one of these four families.
\item[(iii)]  Suppose that  $C$ is not a circle. If
the first fundamental form $ds^2$ of $g$ admits
at most finitely many effective symmetries $($in particular, this assumption 
follows if the singular curvature 
function $\kappa_s$ of $g$ is not constant,
see Remark \ref{rem:Ad}$)$,
then for each choice of $f^i_{P}$ $(i\in \{1,2,3,4\},\,\, P\in C)$, 
\begin{align*}
\Lambda^i_{P}&:=\Big\{(j,Q)\in \{1,2,3,4\}\times C\,;\, 
 \mbox{$f^j_{Q}$ is congruent to $f^i_{P}$}
\Big\}
\end{align*}
is a finite set. In particular, there are 
uncountably many mutually non-congruent isomers of $g$.
\item[(iv)] Suppose $C$ has no symmetries. 
If $ds^2$ does not admit any
effective symmetries
$($in particular, this assumption 
follows if $\kappa_s$ has no symmetries, see Remark \ref{rem:Ad}$)$,
then the set $\Lambda^i_{P}$ is a one-point-set for each
$(i,P)\in \{1,2,3,4\}\times C$.
\end{enumerate}
\end{theorem}

\begin{remark}\label{rem:New}
To show the existence of $\check f$
 as an isomer of $f\in \mc F^\omega(C)$, 
the assumption \eqref{eq:angle}
is needed to
apply the Cauchy-Kowalevski theorem (see \cite[Theorem I]{HNSUY})D
(On the other hand, if \eqref{eq:angle} fails, one can find cuspidal edges whose 
isomers do not exist, see \cite[Corollary 4.13]{HNSUY} for details.)
However, to construct infinitely many isomers of 
$f\in \mc F^\omega(C)$,
the condition \eqref{eq:angle} is not sufficient, and we need to assume that
$f$ must  belong to the class $\mc F^\omega_*(C)$,
see \eqref{eq:Ref} in the proof of
Theorem \ref{thm:main} below.
\end{remark}

\begin{remark}\label{rem:Ad}
We may assume that the initially given $g\in \mc F^\omega_*(C)$ is 
a normal form. Suppose that $ds^2$ admits an 
effective symmetry
$\phi$.
(Later, we will show that any symmetry of $ds^2$ is effective, 
see Corollary \ref{cor:E}.)
Then it induces a symmetry of
the singular curvature function $\kappa_s$ of $g$. 
Hence,
\begin{enumerate}
\item[(1)] {\it the conclusion of {\rm (iv)} follows if $\kappa_s$ has no symmetries}.
\end{enumerate}
Moreover,
\begin{enumerate}
\item[(2)]
{\it the conclusion of {\rm (iii)} follows if
$\kappa_s$ is non-constant}.
\end{enumerate}
In fact, if $ds^2$ admits infinitely many 
distinct effective symmetries,
then they give infinitely many 
symmetries of $\kappa_s$.
Since $\kappa_s$ is real analytic, it must be constant.
\end{remark}

To construct $f^i_{P}$ along the knot $C$, 
the real analyticity of $f$
and the condition (1.3)
are required,
because we need to apply the Cauchy-Kowalevski theorem inductively
(cf. Lemma \ref{lem:ex}). 

Since the curvature function of the curve 
$\gamma_m$ ($m\ge 1$) given in
\eqref{eq:cm} is non-constant,
the cuspidal edges $f_m$  
as in Example \ref{eq;k}
satisfy the condition (iii) in Theorem \ref{thm:main}, and
each $f_m$ has infinitely many isomers.
On the other hand, $f_m$ does not satisfy the condition (iv),
since $C$ admits a symmetry.
We can show the existence of an  example satisfying (iv),
as follows:

\begin{example}
We consider a closed convex $C^\infty$-regular curve $C$, 
lying in the $2$-plane 
$
\R^2:=\{(x,y,0)\in \R^3 \,;\, x,y\in \R\}. 
$
We can choose $C$ so that  it has no symmetries as a plane curve.
Considering the approximation of $C$ by Fourier series,
by  Lemma \ref{lem:A0} in the appendix,
we may assume that $C$ is real analytic.
Let $\pi:S^2\setminus\{(0,0,1)\}\to \R^2$ be 
the stereographic projection. 
We let $\gamma(t)$ $(0\le t\le l)$ be
the arc-length parametrization of $C$ and
set
$$
\tilde \gamma_u(t):=\frac{\pi^{-1}(u\, \gamma(t))+(0,0,1)}{2u},
$$
which is a real analytic $1$-parameter deformation of
$\tilde \gamma_0:=\gamma$ for sufficiently small $|u|$.
We denote by $\tilde C_u$ the image of $\tilde \gamma_u(t)$.
Then $\tilde C_0=C$. 
Since the length $L(u)$ of $\tilde C_u$ depends 
real analytically on $u$, it is 
a  real analytic function of $u$.
So we set
$$
\gamma_u(t):=\frac{l}{L(u)}\tilde \gamma_u(t),
$$
which gives a $1$-parameter family of $C^\omega$-embedded space 
curves of period $l$ satisfying $\gamma_0=\gamma$.
Since $C$ has no symmetries,
the curvature function $\kappa_u$ 
of $\gamma_u$ also has no symmetries
for sufficiently small $|u|$ (cf. Lemma \ref{lem:A0}).
Since $\gamma_u$ ($u\ne 0$) lies on a sphere,
and its curvature function as a space curve is not
constant, it cannot be a part of any circle,
and so its torsion function never vanishes identically.
In particular, it does not lie in any plane.
As a consequence, 
the image of $\gamma_u$ ($u\ne0$) has no symmetries. 
We fix $u$, and parametrize $\gamma_u$ by
arc-length. Then by
Fukui's formula (cf. \eqref{eq:repF}),
we can construct a $C^\omega$-cuspidal edge $f_u$
with constant cuspidal angle $\theta_0\in (0,\pi/2)$
and with $(A(t,v),B(t,v)):=(t^2,t^3)$.
Since the singular curvature function $\kappa_u\cos \theta_0$
of $f_u$ along the $t$-axis
has no symmetries, part (1) of Remark \ref{rem:Ad}
implies that
$f_u$ satisfies (iv) of Theorem \ref{thm:main}.
\end{example}

\section{Proof of Theorem \ref{thm:main}}\label{sec2}

A positive semi-definite $C^\omega$-metric 
$
ds^2=Edt^2+2Fdtdv+Gdv^2
$
defined on $U_\epsilon(S^1)$ is called
a {\it periodic Kossowski metric}  (cf. \cite{HNUY} or \cite{HNSUY})
if it satisfies the following:
\begin{enumerate}
\item[(a)] 
$F(t,0)=G(t,0)=0$,
$E_v(t,0)=2F_t(t,0)$ and $G_t(t,0)=G_v(t,0)=0$, 
\item[(b)] there exists a $C^\omega$-function $\lambda$ defined on
$U_\epsilon(S^1)$
satisfying 
$
EG-F^2=\lambda^2,
$
$\lambda(t,0)=0$ and $\lambda_v(t,0)\ne 0$ for each $t\in S^1$.
\end{enumerate} 
Under the assumption that $\lambda_v(t,0)>0$, the singular curvature 
$\kappa_s$ of $ds^2$
is defined by (cf. \cite[Remark 3.5]{HNSUY})
\begin{equation}\label{eq:2}
\kappa_s(t):=\frac{-F_v(t,0)E_t(t,0)+2E(t,0)
F_{tv}(t,0)-E(t,0)E_{vv}(t,0)}{2E^{3/2}(t,0)\lambda_v(t,0)}
\end{equation}
for each  $t\in S^1$.
If  $ds^2$ is
the first fundamental form of $f\in \mc F^\omega(C)$, 
it is a periodic Kossowski metric
(cf. \cite[Lemma 2.9]{HNSUY}),
and the singular curvature of $f$
has the above expression.

\medskip
\begin{lemma}\label{lem:ex}
Let $\gamma(t)$ $(t\in S^1)$ be a closed $C^\omega$-curve 
embedded in $\R^3$ parametrized by arc-length
whose curvature function $\kappa(t)$ is positive everywhere.
Let $ds^2$ be a periodic Kossowski metric on $U_\epsilon(S^1)$
satisfying
$$
E(t,0)=1,\quad F(t,0)=G(t,0)=0\qquad (t\in S^1).
$$
Suppose that the 
singular curvature $\kappa_s$ along the 
singular curve $S^1\ni t\mapsto (t,0)\in U_\epsilon(S^1)$ satisfies
\begin{equation}\label{eq:mM0}
|\kappa_s(t)|< \kappa(t) \qquad (t\in S^1).
\end{equation}
Then there exists a cuspidal edge $f_+$ 
$($resp. $f_-)$ along $C:=\gamma(S^1)$ 
satisfying
\begin{enumerate}
\item $f_+$ $($resp. $f_-)$ is defined on $U_\delta(S^1)$
for some $\delta\in (0,\epsilon)$,
\item  $f_+(t,0)=\gamma(t)$ $($resp. 
$f_-(t,0)=\gamma(t))$ for each $t\in S^1$,
\item the first fundamental form of $f_+$ $($resp. $f_-)$  is $ds^2$, 
\item 
the limiting normal curvature 
 of $f_+$ $($resp. $f_-)$ is positive-valued $($resp. negative-valued$)$
and is equal to
$$
\sqrt{\kappa(t)^2-\kappa_s(t)^2},\qquad
(\text{resp. }\, -\sqrt{\kappa(t)^2-\kappa_s(t)^2}), 
$$
and
\item $f_+$ and $f_-$ belong to $\mc F^\omega_*(C)$.
\end{enumerate}
Moreover, if there exists a cuspidal edge $g$
defined on an open subset $V(\subset U_\epsilon(S^1))$ 
containing $S^1\times \{0\}$ such that
$g(t,0)=\gamma(t)$ holds for $t\in S^1$
and the first fundamental form of $g$
is $ds^2$, 
then $g$ is right equivalent to $f_+$ or $f_-$.
\end{lemma}

\begin{proof}
We can find a partition
$
0=t_0<t_1<\cdots<t_n=l
$
of $S^1=\R/l\Z$
such that 
there exist local coordinate systems $(U_i;x_i,y_i)$
($i=1,\dots,n$)
of $U_\epsilon(S^1)$ containing 
$[t_{i-1},t_i]\times \{0\}$,
where $n$  is a certain positive integer.
For each $i\in \{1,...,n\}$, the metric has the expression
$
ds^2=E_i(dx_i)^2+G_i(dy_i)^2
$
on $U_i$.
Since $ds^2$ satisfies \eqref{eq:mM0},
we can apply \cite[Theorem 3.9]{HNSUY}
by setting $U:=U_i$ for each $i$, 
and obtain a map
$g_{+,i}:U_i\to \R^3$ 
(resp.  $g_{-,i}:U_i\to \R^3$)
($i=1,\dots,n$)
satisfying (2) and (3) on $U_i$,
and the  limiting normal curvature 
of $g_{\pm,i}$ is equal to
$$
\sqrt{\kappa(t)^2-\kappa_s(t)^2},\qquad
(\text{resp.}\, -\sqrt{\kappa(t)^2-\kappa_s(t)^2}\,\,).
$$
In other words,
the cuspidal angles $\theta_{\pm,i}$ of $g_{\pm,i}$ satisfy
(cf. \eqref{eq:kskn} and also \cite[(3.6)]{HNSUY})
\begin{equation}\label{eq:kn}
\kappa(t)\sin \theta_{+,i}(t)=\sqrt{\kappa(t)^2-\kappa_s(t)^2},
\qquad \theta_{-,i}(t)=-\theta_{+,i}(t).
\end{equation}
Since  conditions (2) and  (3) do not depend on coordinates,
the uniqueness of such a pair of maps yields
that $g_{\pm,i}=g_{\pm,i-1}$ holds on $U_i\cap U_{i-1}$.
So we obtain a map $f_+$ 
(resp. $f_-$) defined on $U_\delta([0,l])$
for a certain $\delta\in (0,\epsilon)$ 
such that each $g_{+,i}$ (resp. $g_{-,i}$)
coincides with $f_+$ (resp. $f_-$)
on $U_\delta([0,l])\cap U_i$ for $i=1,\dots,n$.
In particular, the cuspidal angle functions 
$\theta_{\pm}(t)$ of $f_\pm$ satisfy (cf. \eqref{eq:kn} and \eqref{eq:angle})
$$
\theta_{-}(t)=-\theta_{+}(t),\quad \kappa(t)\cos \theta_{+}(t)=\kappa_s(t),
\quad
\kappa(t)\sin \theta_{+}(t)=\sqrt{\kappa(t)^2-\kappa_s(t)^2}
$$ 
for $t\in S^1$. Then, we obtain
$f_{+}=g_{+,1}$ and $f_{-}=g_{-,1}$
on $U_n\cap U_1$.
In fact, if not, we have $g_{\pm,n}=g_{\mp,1}$ and the function 
$\theta_+(t)$ takes different signs at $t=0$ and $t=l$.
Then by the continuity of $\theta_+(t)$, it has a zero on $[0,l)$, 
a contradiction. So $f_\pm$ are $l$-periodic.
Moreover, $f_+$ and $f_-$ belong to $\mc F^\omega_*(C)$,
because of \eqref{eq:mM}.
The last statement of 
Lemma \ref{lem:ex} follows from the last assertion 
of \cite[Theorem 3.9]{HNSUY}.
\end{proof}

As an application of this lemma, we can prove the following
important conclusion:

\begin{corollary}\label{cor:E}
Let $f\in \mc F^\omega_*(C)$. If the first fundamental form $ds^2_f$ of $f$
has a symmetry $\phi$, then $\phi$ is effective.
\end{corollary}

\begin{proof}
Without loss of generality, we may assume that $f(t,v)$ is
a normal form.
Suppose that $\phi$ is a non-effective symmetry 
of $ds^2_f$. Then $\phi(t,0)=(t,0)$ holds.
If we set $g=f\circ \phi$, then $f(t,0)=g(t,0)$ and 
$g$ has the same first fundamental form as $f$.
Applying Lemma \ref{lem:ex}
 to $ds^2_f$, we obtain two cuspidal edges
$f_\pm$ whose first fundamental forms are $ds^2$
such that $f_\pm(t,0)=f(t,0)$.
Then the last assertion of Lemma \ref{lem:ex}
yields that
either $f=f_+$ or $f=f_-$ holds.
Moreover, applying
the last assertion of Lemma \ref{lem:ex}
again, we can conclude that
$g$ coincides with $f_+$ or $f_-$.
Since $g$ has the same image as $f$,
the cuspidal angle of $g$ at each point of $C$
coincides with that of $f$.
Thus, $g$ must coincide with $f$, which implies
$f\circ \phi=f$.
Since $f$ is a cuspidal edge, it is an injective map.
So $\phi$ must be the identity map.
\end{proof}

\smallskip
\noindent
{\it Proof of  Theorem \ref{thm:main}.}
Without loss of generality, we may assume that 
$g(t,v)$ itself is expressed in
a normal form \and that $\gamma(t)$ is parametrized by arc-length.
Replacing $\gamma(t)$ by $\gamma(\sigma t+b)$
for suitable $\sigma\in \{1,-1\}$ and $b\in S^1$,
we may assume that 
\begin{equation}\label{eq:g-new}
g(t,0)=\gamma(t)\qquad (t\in S^1).
\end{equation}
We denote by $ds^2$ the first fundamental form of $g$.
We now construct 
the four families $f^i_{\gamma(a)}$ 
for each $a\in [0,l)$.
Since $f$ satisfies \eqref{eq:mM},
it holds that
\begin{equation}\label{eq:Ref}
\kappa_s(t)\le 
\max_{u\in S^1}|\kappa_s(u)| < \min_{u\in S^1}\kappa(u)\le \kappa(\sigma t+a)
\qquad (\sigma\in \{1,-1\})
\end{equation}
for $t\in S^1$. In particular, we can apply Lemma \ref{lem:ex} for 
the closed $C^\omega$-curves
$$
t\mapsto \gamma(t+a) \,\, \text{ and }\,\,
t\mapsto \gamma(-t+a) \qquad (a\in [0,l)).
$$
Thus, we obtain four isomers $f^i_{\gamma(a)}$ associated with $g$ 
($i=1,2,3,4$) such that
\begin{itemize}
\item[(1)] each $f^i_{\gamma(a)}$ belongs to 
$\mc F^\omega_*(C)$
whose first fundamental form is $ds^2$,
\item[(2)] $f^j_{\gamma(a)}(t,0)=\gamma(t+a)$ 
($t\in S^1$) holds for $j=1,2$, and
$f^k_{\gamma(a)}(t,0)=\gamma(-t+a)$ 
($t\in S^1$) holds for $k=3,4$, and
\item[(3)]
the limiting normal curvature 
 of $f^i_{\gamma(a)}$ ($i=1,2,3,4$)
is equal to
\begin{equation}\label{eq:kg00}
\sigma'_i\sqrt{\kappa(\sigma_i t+a)^2-\kappa_s(t)^2},
\end{equation}
where $\kappa_s$ is the singular curvature of $g$ along $\gamma$ and
\begin{equation}\label{eq:sigma}
\sigma_i:=
\begin{cases}
1 & \text{if $i=1,2$},\\
-1 & \text{if $i=3,4$}, 
\end{cases}
\qquad
\sigma'_i:=
\begin{cases}
1 & \text{if $i=1,3$}, \\
-1 & \text{if $i=2,4$}.
\end{cases}
\end{equation}
\end{itemize}

\begin{remark}
In the above construction,
each $f^i_{\gamma(a)}$
might not be expressed as a normal form.
We give here such an example.
We consider the helix (cf. \cite[Example 5.4]{HNSUY})
$$
\gamma(t):=
\left(\cos \left(\frac{t}{\sqrt{2}}\right),
\sin \left(\frac{t}{\sqrt{2}}\right),
\frac{t}{\sqrt{2}}\right)
$$
parametrized by arc-length
defined on a certain bounded closed interval
containing $t=0$.
We set
$$
(A(v),B(v)):=\int_0^v w(\cos w,\sin w)dw=(v \sin v+\cos v-1,\sin v-v \cos v),
$$
and fix a constant  $\theta\in (0,\pi/2)$.
Then the map $f(t,v)$ induced by Fukui's formula 
\eqref{eq:repF}
gives a normal form of a cuspidal edge along the
helix. As shown in \cite[Example 5.4]{HNSUY},
we can express the first fundamental form
$ds^2_f$ of $f$ in the following form
$$
ds^2_f=E(v) dt^2+2F(v)dt dv+G(v)dv^2,
$$
where 
$F(v)=v (v -\sin v)/2,\,\,G(v)=v^2$,
and $E(v)$ is
 a certain positive
valued $C^\omega$-function of $t$ which depends
on $\theta$ (one can compute $E(v)$ explicitly using the formula given in
\cite[Proposition 4.8]{HNSUY}).

Also, as shown in \cite[Example 5.4]{HNSUY},
an isomer (the dual) $\check f$ of $f$ can be written as
(the figures of the images of $f$ and $\check f$ are indicated 
in Figure 1, right)
$$
\check f=T\circ f \circ \phi,
$$
where $T$ is the $180^\circ$-rotation with
respect to the principal normal line at $\gamma(0)$
of the helix $\gamma$, and
$\phi$ is an effective symmetry of $ds^2_f$
fixing $(0,0)$. 
(The inverse and the inverse dual of $f$ are given by
 $f_*=T\circ f$ and 
$\check f_*=T\circ \check f$, respectively.)
Such a symmetry $\phi$ must be
uniquely determined by \cite[Proposition 3.15]{HNSUY}.
If $\check f(t,v)$  is also a normal form,
then the map $\phi$ must have
the expression
$
\phi(t,v)=(-t,v).
$
However, this contradicts the fact
that 
$
\phi(x,y)=(-x,y)
$
holds for the local coordinate system at $(0,0)$
given by (cf. \cite[(5.2)]{HNSUY})
$$
x(t,v):=t+\int_0^v \frac{w (w -\sin w)\,}{2E(w)}dw,\quad
y(t,v):=\int_0^v \sqrt{\frac{4E(w)-(w -\sin w)^2}{4E(w)}}\, dw.
$$
\end{remark}

We return to the proof of Theorem \ref{thm:main} and prove (i).
By (1),  each $f^i_{\gamma(a)}$ ($i=1,2,3,4,\,\, a\in S^1$) 
has the same first fundamental form as $g$.
In particular, the singular curvature function of $f^i_{\gamma(a)}$
coincides with that of $g$. Thus, $f^i_{\gamma(a)}$ satisfies 
\eqref{eq:mM} and  belongs to $\mc F^\omega_*(C)$.
By \eqref{eq:g-new} and (3),  we have
$g=f^{1}_{\gamma(0)}$ (resp. $g=f^{2}_{\gamma(0)}$)
if the cuspidal angle of $g$ is positive (resp. negative),
proving all assertions in (i).

We next prove (ii):
Suppose that $h$ is an isomer of $g$.  
Then, there exists a local diffeomorphism $\phi$
such that $\phi^*ds^2_h=ds^2$, where $ds^2_h$ is
the first fundamental form of $h$.
Then $h\circ \phi$ has the same first fundamental form
as $g$.
Since $t\mapsto h\circ \phi(t,0)$ gives an
arc-length parametrization of $C$, 
we can write
$$
h\circ \phi(t,0)=\gamma(\sigma_1 t+b)
$$
for some $\sigma_1\in \{1,-1\}$ and $b\in S^1$.
By the last statement of Lemma \ref{lem:ex},
$h\circ \phi$ coincides with $f_{\gamma(b)}^j$
for some $j\in \{1,2,3,4\}$, proving (ii).

We then prove (iii):
Let 
$
f_n:=f^{j_n}_{\gamma(a_n)}
$
($n=1,2,...$)
be mutually distinct isomers of $g$ 
which are congruent to each other,
where $j_n\in \{1,2,3,4\}$
and $a_n \in [0,l)$.
Replacing $\{f_n\}$ by a 
suitable subsequence  if necessary, 
we may assume that the sequence $\{a_n\}$ consists of distinct values.
By Corollary~\ref{cor:X},
there exist an isometry $T_n$ of $\R^3$ 
and a local diffeomorphism $\phi_n$
such that
\begin{equation}\label{eq:fn}
f_n=T_n\circ f_1\circ \phi_n
\end{equation} 
holds (i.e. $T_1$ and $\phi_1$ are identity maps), 
which implies $\phi^*_n ds^2=ds^2$.
By Corollary~\ref{cor:E},
$\phi_n$ is effective unless it is the identity map.
Since $ds^2$ admits only finitely many effective symmetries,
we may assume that $\phi:=\phi_n$ does not depend on $n$.
Then we have
$$
T_n^{-1}\circ f_n=f_1\circ \phi=T_2^{-1}\circ f_2
$$
for $n\ge 3$.
Substituting $v=0$ 
and using the fact that $f_n=f^{_{j_n}}_{\gamma(a_{n})}$,
we have
$$
\gamma(\sigma_{j_n}^{} t+a_{n})=f_n(t,0)=T_n\circ T_2^{-1}
\circ f_2(t,0)=
T_n\circ T_2^{-1}\circ \gamma(\sigma_{j_2}^{} t+a_{2}),
$$
where $\sigma_{k}$ ($k=1,2,3,...$) are defined
in \eqref{eq:sigma}.
In  particular,
$$
\kappa(\sigma_{j_n}^{} t+a_{n})=\kappa(\sigma_{j_2}^{} t+a_{2}),\qquad
\sigma''_{n}\tau(\sigma_{j_n}^{} t+a_{n}^{})=\tau(\sigma_{j_2}^{} t+a_2)
$$
hold, where (\lq\lq$\det$" denotes the determinant of square matrices)
$$
\sigma''_n:=\det(T_n\circ T_2^{-1}) \in \{1,-1\}
$$
and $\tau(t)$ is the torsion function of $\gamma(t)$.
Substituting $t=0$,  we have
$$
\kappa(a_{n})=\kappa(a_{2}),\qquad
\sigma''_n\tau(a_{n})=\tau(a_{2}).
$$
Since the sequence $\{a_{n}\}$ takes distinct values,
this accumulates to a value $a_\infty\in S^1$.
Since $\kappa(t)$ and $\tau(t)$ are real analytic functions,
they must be constant.
Since $C$ is a knot, it must be a circle lying in a plane,
a contradiction.

Finally, we show (iv):
We fix
$f_0(=f^i_{\gamma(a)})$ ($i\in \{1,2,3,4\}$),
where $a\in [0,l)$. 
Suppose that 
$
f_1=f^{j}_{\gamma(b)}
$
($(j,b)\ne (i,a)$)
is congruent to $f_0$,
where
$j\in \{1,2,3,4\}$
and $b \in [0,l)$.
By Corollary~\ref{cor:X}, there exist an isometry $T$ of $\R^3$
and a local diffeomorphism $\phi$
such that
$ 
f_1(t,v)=T\circ f_0\circ \phi.
$
Since $C$ has no symmetries, $T$ must be the identity map.
Moreover, since $f_1$ and $f_0$ have the same first fundamental
form, $\phi^*ds^2=ds^2$ holds.
Since $ds^2$ does not admit any 
effective symmetries, 
Corollary~\ref{cor:E} yields that
$\phi$ is the identity map.
Hence $f_1=f_0$, and $\Lambda^i_{\gamma(a)}$ ($a\in [0,l)$)
is a one-point set.
\qed

\begin{remark}\label{rmk:4}
When $J=I$, that is, $C$ is non-closed,
the authors showed in \cite{HNSUY} that for each 
germ of a cuspidal edge $f$ satisfying
\begin{equation}\label{eq:mM2}
\max_{t\in I}|\kappa_s(t)| < \min_{t\in I}\kappa(t),
\end{equation}
there exist three isomers $\check f,f_*$ and $\check f_*$ of $f$
(see Remark \ref{rem:I}).
If we set $I_a:=[a-\delta, a+\delta]$
for sufficiently small $\delta>0$,
the restrictions $(f^j_{\gamma(a)})\Big |_{I_a}$ ($j=2,3,4$) of
$f^j_{\gamma(a)}$ (constructed in the above proof)
to the interval $I_a(\subset S^1)$
coincide with these three isomers (cf. \eqref{eq:sigma}).
\end{remark}

\appendix

\section{A property of non-symmetric functions}

We prove the following assertion:

\begin{lemma}\label{lem:A0}
{\it Let $\{\mu_s(t)\}_{s\in [0,l)}$ 
be a continuous one-parameter family of
$C^\infty$-functions on $\R$
satisfying $\mu_s(t+l)=\mu_s(t)$ for each $s$.
If $\mu_0$ has  no symmetries $($cf. Definition \ref{def:curveS}$)$, 
then $\mu_s$ also has no symmetries
for sufficiently small  $s(>0)$.}
\end{lemma}

\begin{proof}
If the assertion fails,
then there exists a monotone decreasing sequence $\{s_n\}$ 
converging to $0$ such that
$\tilde \mu_n:=\mu_{s_n}$ has a certain symmetry,
that is,
there exist a constant 
$c_n\in [0,l)$
and a sign $\sigma_n\in \{1,-1\}$ such that
\begin{equation}\label{eq:cn}
\tilde \mu_n(\sigma_n t+c_n)=\tilde \mu_n(t).
\end{equation}
Since $\sigma_n=\pm 1$, by replacing $\{s_n\}$ by some
subsequence if necessary, we  may assume that 
$
\sigma:=\sigma_n
$
does not depend on $n$.
Since $\R/\Z$ is compact,
replacing $\{s_n\}$ by some
subsequence if necessary, we  may assume that 
$c_n$ converges to $c_0$.
Then taking the limit $n\to \infty$, we have
$
\mu_0(\sigma t+c_0)=\mu_0(t).
$
Since $\mu_0$ is non-symmetric, 
we have $\sigma=1$ and $c_0\in\{0,l\}$. Then, 
we may assume that $c_0=0$ without loss of generality. 
If $c_n$ is an irrational number, 
then $\{\tilde \sigma_n m+c_n\}_{m\in \Z}$
is dense in $S^1$, and
\eqref{eq:cn} yields that $\mu_n$ is a constant 
function. Since $\mu_0$ has no symmetries, 
we may assume that $c_n$ is a rational number
for sufficiently large $n$, and we can write
$
c_n:={q_n}/{p_n},
$
where $p_n$ and $q_n$ are relatively prime integers. 
Then there is a pair $(a,b)$ of integers such that
$
ap_n+bq_n=1
$
and
\begin{equation}\label{eq:x}
\mu_n(t)=\mu_n(t+bc_n)=\mu_n\!\!\left(t+\frac{1-ap_n}{p_n}\right)
=\mu_n\!\!\left(t+\frac{1}{p_n}
\right).
\end{equation}
Fix an irrational number $x_0\in (0,1)$.
Then there exist integers $r_n$
($n=1,2,3,...$) so that
$x_n:=r_n/p_n$ 
converges to $x_0$.
Since 
$
\mu_n(t)=\mu_n(t+x_n)
$
by \eqref{eq:x}, taking the limit as $n\to \infty$,
we have $
\mu_0(t)=\mu_0(t+x_0)
$,
contradicting the assumption that $\mu_0(t)$ has no symmetries. 
\end{proof}

\begin{acknowledgments}
The authors thank the referee for valuable comments.
\end{acknowledgments}

\end{document}